\documentclass{amsproc}

\usepackage{amsmath, amsthm, amssymb}
\newtheorem{theorem}{Theorem}
\newtheorem{lemma}[theorem]{Lemma}
\newtheorem{corollary}[theorem]{Corollary}
\newtheorem{definition}[theorem]{Definition}

\begin{document}

\title{Generalizations of product-free subsets}
\author{Kiran S. Kedlaya}
\address{Department of Mathematics, Massachusetts Institute of Technology,
77 Massachu\-setts Avenue, Cambridge, MA 02139}
\email{kedlaya@mit.edu}
\urladdr{http://math.mit.edu/~kedlaya/}

\author{Xuancheng Shao}
\address{410 Memorial Drive Room 211A, Cambridge, MA 02139}
\email{zero@mit.edu}

\thanks{The first author was supported by NSF CAREER grant DMS-0545904 and a Sloan Research Fellowship. The second author was supported by the Paul E. Gray (1954) Endowed Fund for the MIT Undergraduate Research Opportunities Program.}

\begin{abstract}

In this paper, we present some generalizations of Gowers's result about product-free subsets of groups. For any group $G$ of order $n$, a subset $A$ of $G$ is said to be product-free if there is no solution of the equation $ab=c$ with $a,b,c\in A$. Previous results showed that the size of any product-free subset of $G$ is at most $n/\delta^{1/3}$, where $\delta$ is the smallest dimension of a nontrivial representation of $G$. However, this upper bound does not match the best lower bound. We will generalize the upper bound to the case of {\em product-poor} subsets $A$, in which the equation $ab=c$ is allowed to have a few solutions with $a,b,c \in A$. We prove that the upper bound for the size of product-poor subsets is asymptotically the same as the size of product-free subsets. We will also generalize the concept of product-free to the case in which we have many subsets of a group, and different constraints about products of the elements in the subsets.

\end{abstract}

\maketitle

\section{Background}

Let $G$ be a group. A subset $S$ of $G$ is {\em product-free} if there do not exist $a,b,c \in S$ (not necessarily distinct) such that $ab=c$. 
One can ask about the existence of large product-free subsets for various groups, such as the groups of integers or compact topological groups.
Assume that $G$ is a finite group of order $n>1$. Let $\alpha (G)$ denote the size of the largest product-free subset of $G$, and $\beta (G) = \alpha (G) /n$.
The purpose is to find good bounds for $\alpha (G)$ and $\beta (G)$ as a function of $G$, or as a function of $n$.

If $G$ is abelian, this problem was solved by Green and Ruzsa in 2005 \cite{Green05}. They gave an exact value of $\alpha (G)$ as a function of some characteristics of the abelian group $G$. However, the problem is much harder for nonabelian $G$. The first appearance of the problem of computing $\alpha(G)$ for nonabelian $G$ seems to have been in a 1985 paper \cite{Babai85}. The construction of product-free subsets given in \cite{Babai85} is quite simple: if $H$ is a proper subgroup of $G$, then any nontrivial coset of $H$ is product-free. Therefore, a lower bound for $\beta(G)$ can be derived: $\beta(G) \geq m^{-1}$, where $m$ is the index of the largest proper subgroup of $G$.
In 1997, Kedlaya \cite{Kedlaya97} improved this bound to $cm^{-1/2}$ for some constant $c$ by showing that if $H$ has index $k$ then one can in fact find a union of $ck^{1/2}$
cosets of $H$ that is product-free. This gives the best known lower bound on $\alpha(G)$ for general $G$.

\begin{theorem}
\label{thm:low}

$\beta(G) \geq cm^{-1/2}$ for some constant $c$.

\end{theorem}

On the other hand, Gowers recently established a nontrivial upper bound for $\alpha(G)$ using a remarkably simple argument \cite{Gowers} (see also
\cite{kedlaya2007pfs} in this volume). The strategy of Gowers
is to consider
three sets $A,B,C$ for which there is no solution of the equation $ab = c$ with $a \in A$, $b \in B$, $c \in C$, and give an upper bound on $|A| \cdot |B| \cdot |C|$, 
where $|A|$ denote the order of the set $A$. 

\begin{theorem} 
\label{thm:Gowers}

If $A$, $B$, $C$ are subsets of $G$ such that there is no solution of the equation $ab = c$ with $a \in A$, $b \in B$, $c \in C$, then $|A| \cdot |B| \cdot |C| \leq n^3/\delta$,
where $\delta$ is defined as the smallest dimension of a nontrivial representation of $G$. Consequently, $\beta(G) \leq \delta ^{-1/3}$.

\end{theorem}

To prove this, a certain {\em bipartite Cayley graph} $\Gamma$ associated to $G$ is constructed. The vertex set of $\Gamma$
is $V_{1} \cup V_{2}$, where each $V_{i}$ is a copy of $G$, with an edge from $x \in V_{1}$ to $y \in V_{2}$ if and only if $yx^{-1} \in A$. Therefore,
there are no edges between $B \subseteq V_{1}$ and $C \subseteq V_{2}$. Let $N$ be the {\em incidence matrix} of $\Gamma$, with columns indexed by $V_{1}$ and rows
by $V_{2}$, with an entry in row $x$ and column $y$ if $xy$ is an edge of $\Gamma$. Define $M = NN^{\mathrm{T}}$; then the largest eigenvalue of $M$ is $|A|$.  
Let $\lambda$ denote the  second largest eigenvalue of the matrix $M$. Gowers's theorem relies on the following two lemmas.

\begin{lemma}
\label{lemma:1}

The second largest eigenvalue $\lambda$ of the matrix $M$ is at most $n \cdot |A| / \delta$.

\end{lemma}

\begin{lemma}
\label{lemma:2}

$|A| \cdot |B| \cdot |C| \leq n^2 \cdot \frac{\lambda}{|A|}$.

\end{lemma}

It can be proved that the group $G$ has a proper subgroup of index at most $c \delta^2$ for some constant $c$ \cite{Liebeck96}. 
Therefore, we have the following bounds for $\alpha(G)$.

\begin{theorem} 
\label{thm:main}

$cn/\delta \leq \alpha(G) \leq cn/\delta^{1/3}$ for some constant $c$.

\end{theorem}

Since the gap between the lower bound and the upper bound for $\alpha(G)$ appears quite small, one might ask about closing it. However, it has been proved in \cite{kedlaya2007pfs} that Gowers's argument alone is not sufficient since the upper bound in Theorem \ref{thm:Gowers} cannot be improved if the three sets $A$, $B$, and $C$ are allowed to be different. In addition, Gowers also made some generalizations to the case of many subsets. Instead of finding two elements $a$ and $b$ in two subsets $A$ and $B$ such that their product is in a third subset $C$, he proposed to find $x_1, \ldots, x_m$ in $m$ subsets such that for every nonempty subset $F \subset \{1,2,\ldots,m\}$, the product of those $x_i$ with $i \in F$ lies in a specified subset.

In this paper, instead of trying to close the gap between the lower bound and the upper bound for $\alpha(G)$, we will show that Gowers's Theorem \ref{thm:Gowers} can actually be generalized to {\em product-poor subsets} of a group. We will give the precise definition of product-poor subsets as well as the upper bound for the size of product-poor subsets in Section \ref{sec:product-poor}. In Section \ref{sec:manysubsets}, we will examine Gowers's generalization to the case of many subsets, and we will further generalize it to the problem of finding $x_1, \ldots, x_m$ in $m$ subsets such that for certain ({\em not} all) subsets $F \subset \{1,2,\ldots,m\}$ the product of those $x_i$ with $i \in F$ lies in a specified subset.

\section{Product-poor subsets of a group}
\label{sec:product-poor}

In this section, we will state and prove a generalization of Theorem~\ref{thm:Gowers}. We will consider the size of the largest product-poor subset instead of product-free subset. In product-poor subsets, there are a few pairs of elements whose product is also in this set. It turns out that we can derive the same asymptotic upper bound for the size of a product-poor subset. Despite the fact that the best known lower bound and the upper bound for the size of the largest product-free subset do not coincide, these two bounds do coincide asymptotically for the largest product-poor subset. First of all, we give the precise definition of a product-poor subset.

\begin{definition}

A subset $A$ of group $G$ is $p$-product-poor iff the number of pairs $(a,b)\in A\times A$ such that $ab\in A$ is at most $p|A|^2$.

\end{definition}

We now give a generalization of Gowers's argument. We change the condition that there are no solutions of the equation $ab=c$ with $a\in A$, $b\in B$, $c\in C$, to the weaker condition that there are only a few solutions of that equation. In fact, Babai, Nikolov, and Pyber \cite{BabaiNikolovPyber} recently discovered a more general result depending on probability distributions in $G$. However, here we want to emphasize the result that the upper bound of $|A| \cdot |B| \cdot |C|$ is asymptotically the same for both product-free subsets and product-poor subsets. Note that this theorem is similar to a lemma in \cite{Gowers}, which is stated in Lemma~\ref{lemma:Gowers} below.

\begin{theorem}
\label{thm:product-poor}

Let $G$ be a group of order $n$. Let $A$, $B$, and $C$ be subsets of $G$ with orders $rn$, $sn$, and $tn$, respectively. If there are exactly $prstn^2$ solutions of the equation $ab = c$ with $a \in A$, $b \in B$, $c \in C$, then $rst(1-p)^2\delta \leq 1$.

\end{theorem}

\begin{proof}

Let $\mathbf{v}$ be the characteristic function of $B$, and put $\mathbf{w} = \mathbf{v} - s \mathbf{1}$. Then
\begin{align*}
\mathbf{w} \cdot \mathbf{1} &= 0, \\
\mathbf{w} \cdot \mathbf{w} &= (1-s)^2 \cdot sn + s^2(n-sn) = s(1-s)n \leq sn,
\end{align*}
so by Lemma \ref{lemma:1}, $\Vert N \mathbf{w} \Vert ^2 \leq rn^2 sn/\delta$.

Let $N$ be the incidence matrix of $G$ of size $n \times n$, in which there is an entry in row $x$ and column $y$ iff $xy^{-1} \in B$. Consider the submatrix $N_1$ of $N$ containing those rows corresponding to the elements in $C$, and those columns corresponding to the elements in $A$. The matrix $N_1$ has size $tn \times rn$. Suppose that there are $k_i$ ones in row $i$ of $N_1$ ($1 \leq i \leq tn$). Note that there exists a one-to-one correspondence between the solutions of the equation $ab=c$ with $a \in A$, $b \in B$, $c \in C$, and the nonzero entries in $N_1$. As a result, we have
\begin{displaymath} k_1 + k_2 + \ldots + k_{tn} = prstn^2. \end{displaymath}

Using the above equality, we have
\begin{align*}
\Vert N \mathbf{w} \Vert ^2 & \geq (k_1 - rsn)^2 + (k_2 - rsn)^2 + \ldots + (k_{tn} - rsn)^2 \\
                            & = \sum_{i=1}^{tn} k_i^2 - 2rsn \sum_{i=1}^{tn} k_i + (rsn)^2 \cdot (tn) \\
                            & \geq \frac{1}{tn} (\sum_{i=1}^{tn} k_i)^2 - 2rsn \sum_{i=1}^{tn} k_i + (rsn)^2 \cdot (tn) \\
                            & = p^2 r^2 s^2 t n^3 - 2rsn prstn^2 + r^2 s^2 t n^3 \\
                            & = r^2 s^2 t n^3 (1-p)^2
\end{align*}

Combining the upper bound and the lower bound for $\Vert N \mathbf{w} \Vert ^2$ derived above, we get
\begin{displaymath} r^2 s^2 t n^3 (1-p)^2 \leq \Vert N \mathbf{w} \Vert ^2 \leq rn^2 sn/ \delta, \end{displaymath}
hence $rst(1-p)^2 \delta \leq 1$ as desired.

\end{proof}

Applying Theorem \ref{thm:product-poor} to the special case in which $A=B=C$, we get the following corollary about product-poor subsets of a group.

\begin{corollary}
\label{cor:product-poor}

The size of any $p$-product-poor subset of group $G$ is at most $c|A|/\delta^{1/3}$ for some constant $c=c(p)$ if $p$ is at most $1/\delta^{1/3}$.

\end{corollary}

On the other hand, we can construct a subset $S$ of $G$ by taking the union of $k$ cosets of a nontrivial subgroup of smallest index $m$. We estimate in Theorem \ref{thm:product-poor-construct} an upper bound for the number $p$, such that the subset $S$ is guaranteed to be $p$-product-poor. We do this by estimating the number of solutions of the equation $ab=c$ with $a,b,c \in S$. It turns out that if the value of $k$ is chosen suitably ($k \sim m/\delta^{1/3}$), then the subset $S$ constructed this way has order $\sim n/m^{1/3}$, and this subset $S$ is $1/m^{1/3}$-product-poor. Therefore, for those groups with $m = \Theta(\delta)$ (such as $PSL_2(q)$), we can find a $c/\delta^{1/3}$-product-poor subset of order $\sim n/\delta^{1/3}$ for some constant $c$.

\begin{theorem}
\label{thm:product-poor-construct}

For $G$ admitting a transitive action on $\{ 1,2, \ldots, m \}$ with $m>3$, define
\begin{displaymath} S = \bigcup_{x \in T} \{ g \in G: g(1) = x \}  \end{displaymath}
for any subset $T$ of $\{2,\ldots,m\}$ of order $k \geq 3$. There exists a $T$, such that the number of solutions of the equation $ab=c$ with $a,b,c \in S$ is at most $4n^2 k^3/(m-3)^3$.

\end{theorem}

\begin{proof}

Consider the following summation:
\begin{align*}
& \sum_{T} | \{a, b \in G: a(1), b(1), ab(1) \in T \} |  \\
= & \sum_{a,b} | \{T: |T|=k, a(1),b(1),ab(1) \in T \} | \\
\leq & \sum_{a(1) = b(1)} \binom{m-3}{k-2} + \sum_{a(1)=ab(1)} \binom{m-3}{k-2} + \sum_{b(1)=ab(1)} \binom{m-3}{k-2} + \sum_{a,b} \binom{m-4}{k-3} \\
= & 3 \cdot n \cdot \frac{n}{m} \cdot \binom{m-3}{k-2} + n^2 \cdot \binom{m-4}{k-3} \\
\leq & 4n^2 \cdot \binom{m-4}{k-3}.
\end{align*}

Since there are a total of $\binom{m-1}{k}$ choices for $T$, there exists some $T$ of order $k$, such that the number of solutions of $ab=c$ with $a,b,c \in T$ is at most
\begin{displaymath} 4n^2 \cdot \frac{\binom{m-4}{k-3}}{\binom{m-1}{k}} = 4n^2 \cdot \frac{k(k-1)(k-2)}{(m-1)(m-2)(m-3)} \leq \frac{4n^2 k^3}{(m-3)^3}  \end{displaymath}

\end{proof}

In Theorem \ref{thm:product-poor-construct}, if we choose $k\sim m/ \delta^{1/3}$, then the subset $S$ has order $n/ \delta^{1/3}$, and is $c/\delta^{1/3}$-product-poor for some constant $c$. This shows that the upper bound in Corollary \ref{cor:product-poor} is asymptotically the best for families of groups with $m \sim \delta$. 

\section{Generalization to many subsets}
\label{sec:manysubsets}

In this section, we will study the problem of finding $x_1, \ldots, x_m$ in $m$ subsets such that for subsets $F$ in $\Gamma$, where $\Gamma$ is a collection of subsets of $\{1,2,\ldots,m\}$, the product of those $x_i$ with $i \in F$ lies in a specified subset. Gowers has already studied this problem in which $\Gamma$ is the collection of {\em all} subsets of $\{1,2,\ldots,m\}$. We will first state and prove this result for the simple case when $m=3$, which was first proved in \cite{Gowers}. However, a constant factor is improved significantly compared to the result in \cite{Gowers}, although the proofs are very similar. We will study the situation in which $\Gamma$ is the collection of all two-element subsets of $\{1,2,\ldots,m\}$. After this, it will be easy to derive a result for an arbitrary collection $\Gamma$.

Before we prove any result, we first repeat the following lemma in \cite{Gowers}, which is frequently used in the proofs of all the following results.

\begin{lemma}
\label{lemma:Gowers}

Let $G$ be a group of order $n$ such that no nontrivial representation has dimension less than $\delta$. Let $A$ and $B$ be two subsets of $G$ with densities $r$ and $s$, respectively and let $k$ and $t$ be two positive constants. Then, provided that $rst\geq (k^2\delta)^{-1}$, the number of group elements $x\in G$ for which $|A\cap xB|\leq (1-k)rsn$ is at most $tn$. 

\end{lemma}

We now give an improved version of the corresponding theorem in \cite{Gowers} for the case $m=3$. The constant $M$ is $16$ in \cite{Gowers}, which is much larger than our constant $3+2\sqrt{2}$.

\begin{theorem}
\label{thm:m=3}

Let $G$ be a group of order $n$ such that no nontrivial representation has dimension less than $\delta$. Let $A_1$, $A_2$, $A_3$, $A_{12}$, $A_{13}$ and $A_{23}$ be subsets of $G$ of densities $p_1$, $p_2$, $p_3$, $p_{12}$, $p_{13}$ and $p_{23}$, respectively. Then, provided that $p_1p_2p_{12}$, $p_1p_3p_{13}$ and $p_2p_3p_{23}p_{12}p_{13}$ are all at least $M/\delta$, where $M>3+2\sqrt{2}$, there exist elements $x_1\in A_1$, $x_2\in A_2$ and $x_3\in A_3$ such that $x_1x_2\in A_{12}$, $x_1x_3\in A_{13}$ and $x_2x_3\in A_{23}$.

\end{theorem}

\begin{proof}                                   

Choose two numbers $0 < \lambda < 1$ and $0 < \mu < 1/2$ such that
\begin{displaymath} M \geq \frac{1}{\mu \lambda^2}, M > \frac{1}{(1-\lambda)^2}. \end{displaymath}
(In fact, since $M>3+2\sqrt{2}$, we can choose $\lambda = 2-\sqrt{2}$ and $(3+2\sqrt{2})/(2M) < \mu < 1/2$.) If $p_1 p_2 p_{12} \geq M/\delta$, then we have
\begin{displaymath} p_1 p_2 p_{12} \geq \frac{1}{\mu \lambda^2 \delta} \end{displaymath}
from the choices of $\lambda$ and $\mu$. By Lemma~\ref{lemma:Gowers}, the number of $x_1$ such that 
\begin{displaymath} |A_2\cap x_1^{-1}A_{12}|\leq (1-\lambda) p_2p_{12}n \end{displaymath}
is at most $\mu p_1n$. Similarly, if $p_1p_3p_{13}\geq M/\delta$, then the number of $x_1$ such that 
\begin{displaymath} |A_3\cap x_1^{-1}A_{13}|\leq (1-\lambda) p_3p_{13}n \end{displaymath} 
is also at most $\mu p_1n$. Therefore, since $\mu < 1/2$, we can choose $x_1\in A_1$ such that, setting $B_2=A_2\cap x_1^{-1}A_{12}$ and $B_3=A_3\cap x_1^{-1}A_{13}$, $q_2=(1-\lambda)p_2p_{12}$ and $q_3=(1-\lambda)p_3p_{13}$, we have $|B_2|\geq q_2n$ and $|B_3|\geq q_3n$.

What remains is to show that there exist $y_2\in B_2$, $y_3\in B_3$ and $y_{23}\in A_{23}$, such that $y_2y_3=y_{23}$. This is true from the following relation:
\begin{align*}
q_2 q_3 p_{23} &= p_2p_3p_{23}p_{12}p_{13}(1-\lambda)^2 \\
&\geq \frac{M}{\delta} \cdot (1-\lambda)^2 \\
&\geq \frac{1}{\delta}.
\end{align*}

\end{proof}

In order to generalize Theorem \ref{thm:m=3} to arbitrary $m$, Gowers introduced a concept of {\em density condition} in his paper \cite{Gowers}. The basic idea is that the product of the densities of some subsets must be greater than a threshold. For example, in Theorem \ref{thm:m=3}, the threshold is $M/\delta$, and three products of densities must be greater than this threshold. In order to study the case that $\Gamma$ is the collection of all two-element subsets for general $m$, we do need to define what density products are required to be larger than the threshold. We give the following definition about density product. Although we will not repeat the definition of density condition in \cite{Gowers}, it is worth noticing that the definition given below is intrinsically the same as Gowers's definition of density condition. This similarity can be seen if we set the densities of subsets not considered in our problem to be $1$.

\begin{definition}

Suppose that for every $1\leq i \leq m$ we have a subset $A_i$ of $G$ with density $p_i$, and for every $1\leq i < j \leq m$ we have a subset $A_{ij}$ of $G$ with density $p_{ij}$. 
For any $1\leq i < j \leq m$, we define the $(i,j)$-density product to be 
\begin{displaymath} P_{ij} = p_i p_j p_{ij} \prod_{k=1}^{i-1} (p_{ki} p_{kj}) \end{displaymath}

\end{definition}

Now we state a theorem similar to Gowers's generalization. If all the density products are larger than a threshold $f(m)/\delta$, we are able to find $m$ elements $x_1,\dots,x_m$ from $m$ subsets, such that for any two-element subset $F$ of $\{1,2,\ldots,m\}$, the product of $x_i$ with $i \in F$ is in a specified subset.

\begin{theorem}
\label{thm:manysubsets}

Suppose $f(m)$ is a function such that $f(2) > 1$ and
\begin{displaymath} \left(1-\sqrt{\frac{m}{f(m)}}\right)^2 f(m) \geq f(m-1) \end{displaymath}
for any $m=3,4,\ldots$. Let $G$ be a group of order $n$ such that no nontrivial representation has dimension less than $\delta$. For every $1\leq i \leq m$ let $A_i$ be a subset of $G$ with density $p_i$. For every $1\leq i < j \leq m$ let $A_{ij}$ be a subset of $G$ with density $p_{ij}$. Suppose that all the $(i,j)$-density products are at least $f(m)/\delta$. Then there exist elements $x_1,\ldots,x_m$ of $G$ such that $x_i x_j \in A_{ij}$ for every $1\leq i < j \leq m$.

\end{theorem}

\begin{proof}

We use induction on $m$. Since all the $(1,j)$-density products are at least $f(m)/\delta$, we have the inequality $1/m \cdot p_1 p_j p_{1j} \geq f(m)/(\delta m)$. Therefore, by Lemma~\ref{lemma:Gowers}, for each $1 < j \leq m$, the number of $x_1$ such that 
\begin{displaymath} |A_j \cap x_1^{-1} A_{1j}| \leq p_j p_{1j} \left(1-\sqrt{\frac{m}{f(m)}}\right) \end{displaymath}
is at most $p_1 n / m$. It follows that there exists $x_1 \in A_1$ such that, if for every $1<j \leq m$ we set $B_j = A_j \cap x_1^{-1} A_{1j}$, then every $B_j$ has density at least 
\begin{displaymath} q_j = p_j p_{1j} \left(1-\sqrt{\frac{m}{f(m)}}\right). \end{displaymath}

In order to use the induction hypothesis, it suffices to prove that for every $1<i<j \leq m$, we have
\begin{displaymath} Q_{ij} = q_i q_j p_{ij} \prod_{k=2}^{i-1} (p_{ki} p_{kj}) \geq \frac{f(m-1)}{\delta}. \end{displaymath}
In fact,
\begin{align*}
Q_{ij} &= \left(1-\sqrt{\frac{m}{f(m)}}\right)^2 \cdot p_i p_j p_{1i} p_{1j} \prod_{k=2}^{i-1} (p_{ki} p_{kj}) \\
&= \left(1-\sqrt{\frac{m}{f(m)}}\right)^2 \cdot p_i p_j \prod_{k=1}^{i-1} (p_{ki} p_{kj}) \\
&\geq \left(1-\sqrt{\frac{m}{f(m)}}\right)^2 \cdot \frac{f(m)}{\delta} \\
&\geq \frac{f(m-1)}{\delta}.
\end{align*}

This proves the inductive step of the theorem. We take $m=2$ as the base case, which is trivial since $f(2) > 1$. 

\end{proof}

For example, if we define $f(m)$ by $f(2) = 2$ and 
\[
f(m) = \frac{(m+1-f(m-1))^2}{4m} \qquad (m \geq 3),
\]
then $0 \leq f(m) \leq m+2$ by induction on $m$, and
\begin{displaymath} \left(1-\sqrt{\frac{m}{f(m)}}\right)^2 f(m) = f(m-1). \end{displaymath}
Consequently, the conclusion of Theorem~\ref{thm:manysubsets} holds if we assume
all the density products are at least $(m+2)/\delta$.

Using the same technique as above, we can derive similar results for arbitrary $\Gamma$. There are two aspects we need to consider. One of them is what density products should be involved in the condition. As mentioned before, the definition of density condition in \cite{Gowers} can be used if we set the densities of the sets we are not going to consider to be $1$. Here we give a formal definition of density products (which is similar in the definition in \cite{Gowers}).

\begin{definition}

Suppose $\Gamma$ is a collection of subsets of $\{1,2,\ldots,m\}$. For each $F \in \Gamma$, $A_F$ is a subset of $G$ with density $p_F$. Let $h$ be an integer less than $m$ and let $E$ be a subset of $\{h+1,\ldots,m\}$. Let $P_{h,E}$ be the collection of all the sets in $\Gamma$ of the form $U \cup V$, where $\max{U} < h$ and $V$ is either $\{h\}$, $E$ or $\{h\}\cup E$. We say that the product $\prod_{F\in P_{h,E}}p_F$ is the $(h,E)$-density product. 

\end{definition}

We also need to determine the lower bound for the density products. This lower bound is generally of the form $f_{\Gamma}(m)/\delta$. The constraint of the function $f_{\Gamma}(m)$ is related to the collection $\Gamma$. Using the same method as in the proof of Theorem \ref{thm:manysubsets}, we get the following theorem for arbitrary collection $\Gamma$.

\begin{theorem}

Let $G$ be a group of order $n$ such that no nontrivial representation has dimension less than $\delta$. Let $\Gamma$ be a collection of subsets of $\{1,2,\ldots,m\}$. For each $F\in \Gamma$, let $A_F$ be a subset of $G$ of density $p_F$. Let $h(m)$ be a positive integer such that there are at most $h(m)$ subsets in $\Gamma$ containing $k$ for any $1\leq k \leq m$. Suppose all the $(h,E)$-density products are at least $f_{\Gamma}(m)/\delta$, where the function $f_{\Gamma}(m)$ satisfies the inequality
\begin{displaymath} \left(1-\sqrt{\frac{h(m)}{f_{\Gamma}(m)}}\right)^2 f_{\Gamma}(m) \geq f_{\Gamma}(m-1) \end{displaymath} 
for any $m=2,3,\ldots$, then there exist elements $x_1,\ldots,x_m$ of $G$ such that $x_F\in A_F$ for every $F\in\Gamma$, where $x_F$ stands for the product of all $x_i$ such that $i\in F$.

\end{theorem}

Since the proof is very similar to that of Theorem \ref{thm:manysubsets}, we will not go into details here.

\end{document}